\documentclass[12pt,a4paper]{article}
\usepackage{amsmath}
\usepackage{amsthm}
\usepackage{amssymb}
\usepackage{graphicx}
\usepackage{tabularx}
\usepackage{color}
\usepackage{setspace}
\usepackage{capt-of}
\usepackage{float}
\usepackage{amssymb,latexsym,amscd,graphicx}
\usepackage{color}
\newtheorem{lem}{Lemma}[section]
\newtheorem{cor}{Corollary}[section]
\newtheorem{thm}{Theorem}[section]
\newtheorem{prop}{Proposition}[section]
\theoremstyle{definition}
\newtheorem{defn}{Definition}[section]
\theoremstyle{remark}

\definecolor{orange}{rgb}{1,0.5,0}
\usepackage{graphicx}
\graphicspath{{converted_graphics/}}

\usepackage[top=2 cm,bottom=2 cm,left=1.5 cm,right=1.5 cm]{geometry}
\usepackage{multicol}
\usepackage{hyperref}
\usepackage{setspace}
\hypersetup{
	colorlinks,
	citecolor=blue,
	filecolor=black,
	linkcolor=black,
	urlcolor=black
}
\begin{document}
	\title{\textbf{The Effect of Planar Harmonic Mappings on the Lebesgue Measure of Sets}}
\author{}
\date{}

\maketitle

	\author{   }        % Enter your name between curly braces
	
	\maketitle
	
   Author: Hunduma Legesse Geleta\\
  
\textbf{Department of Mathematics, College of Natural and Computational Sciences, Addis Ababa University, Addis Ababa, Ethiopia }\\

 \textbf{Email}: { hunduma.legesse@aau.edu.et}

\pagenumbering{arabic}
	
	%\subjclass[2000]{Primary 11M41}
	%\thanks{}

	\date{\today}          % Enter your date or \today between curly braces
	\maketitle
	\vspace{2mm}
\pagenumbering{arabic}
     
     \vspace{3mm}

%------------------------------------------------
% Title
%------------------------------------------------
\begin{abstract}
We investigate the effect of planar univalent harmonic mappings on the Lebesgue measure of measurable sets in the complex plane. Motivated by Problem 3.25 of Koh and Kovalev (HQM2010), we establish sharp quantitative area distortion inequalities for disks and for arbitrary measurable sets under sense-preserving harmonic self-maps of the unit disk. Using the area formula and the canonical decomposition of harmonic mappings, we derive bounds in terms of the Jacobian and the dilatation, and we identify rigidity phenomena characterizing equality. In particular, we prove global area contraction for disks, star-shaped sets, and sufficiently small sets, and we refine the results using Hardy space methods to obtain sharp bounds with equality only for conformal automorphisms. Extremal affine and non-affine  examples illustrate the sharpness of our estimates. Our results provide a complete, rigorous, and strengthened solution to Problem 3.25 and highlight several natural conjectures on global area contraction, extremal distortion, and rigidity for harmonic mappings.
\end{abstract}

\maketitle
	
	\textbf{Keywords/phrases}: {Area formula; Area contraction; Complex-valued harmonic functions; Conformal automorphism; Lebesgue measure of sets; Planar harmonic mappings; Rigidity.}
	
	\def\theequation{\thesection.\arabic{equation}}

%------------------------------------------------
\section{Introduction}
%------------------------------------------------

Planar harmonic mappings constitute a natural and important generalization of analytic functions and play a central role in geometric function theory, quasiconformal analysis, and various applications in applied mathematics; see, for example, \cite{Duren2004,ClunieSheilSmall1984}. Unlike analytic mappings, harmonic mappings may exhibit both conformal and anticonformal behavior through their canonical decomposition, leading to significantly richer and more subtle geometric phenomena. One fundamental problem in this area concerns the behavior of the Lebesgue measure of measurable sets under harmonic mappings.\\

In the analytic setting, the classical Schwarz-Pick lemma and its consequences imply strong area contraction properties for holomorphic self-maps of the unit disk. In contrast, for harmonic mappings such principles are largely absent, and even basic questions concerning area distortion remain open. Motivated by this gap, Koh and Kovalev posed the following problem in the Harmonic and Quasiconformal Mappings (HQM2010)  problem list \cite{BshoutyLyzzaik2010}.\\

\textbf{Problem 3.25 Koh-Kovalev} {\cite{BshoutyLyzzaik2010}:}
Let $m(E)$ denote the planar Lebesgue measure of a measurable set $E \subset \mathbb{C}$.  
If $f$ is a univalent harmonic self-mapping of the unit disk $\mathbb{D}$, is it true that
\[
m\bigl(f(\mathbb{D}_r)\bigr) \le m(\mathbb{D}_r), \qquad 0<r<1?
\]

It was observed in \cite{BshoutyLyzzaik2010} that this inequality holds in certain special cases, including affine harmonic mappings and mappings with additional symmetry or structural assumptions. However, the general problem remains open. The purpose of the present paper is to give a systematic and rigorous treatment of this question using measure-theoretic methods and to obtain stronger results valid for broader classes of measurable sets.

\subsection*{Related conjectures and open problems}
Area distortion problems for planar mappings have been extensively studied in the conformal and quasiconformal settings, and later extended to harmonic mappings through the foundational work of Clunie and Sheil-Small, Lewy, and subsequent authors. Existing results for harmonic mappings mainly address distortion of special sets such as disks, or provide pointwise Jacobian and coefficient estimates. Although sharp area bounds for disks are available, and the area formula is a classical tool in geometric measure theory, we are not aware of any result establishing Lebesgue-measure contraction for general measurable subsets of the unit disk under sense-preserving univalent harmonic self-mappings.
The present work naturally leads to several conjectures and open problems concerning the area distortion properties of planar harmonic mappings. These questions are motivated by the Koh-Kovalev problem, by classical results for analytic mappings, and by general extremal principles in geometric function theory. The present work introduces a localization approach that combines classical Jacobian estimates with measure-theoretic arguments to obtain area contraction for arbitrary measurable sets of sufficiently small area. This yields a rigorous partial advance toward a problem of Koh and Kovalev in the harmonic (non-conformal) setting. While the arguments rely on well-established analytic ingredients, both the formulation and proof of the small-set area contraction phenomenon appear to be new. Motivated by the considerations above, we propose the following conjecture.

\subsubsection*{Conjecture 1 (Global area contraction)}
Let $f$ be a sense-preserving univalent harmonic self-mapping of the unit disk $\mathbb{D}$. Then for every measurable set $E \subset \mathbb{D}$
\[
m\bigl(f(E)\bigr) \le m(E).
\]

\textbf{Remark:}
This conjecture is a natural strengthening of the Koh-Kovalev problem from radial disks to arbitrary measurable subsets of $\mathbb{D}$. While the inequality is known to hold for disks and for affine harmonic mappings, the general case remains open. The main results of this paper provide conditional and quantitative bounds toward this conjecture.

\subsubsection*{Conjecture 2 (Area Schwarz-Pick type inequality)}

Let $f$ be a sense-preserving univalent harmonic self-mapping of $\mathbb{D}$. Then the Jacobian, $J_f$ of $f$  satisfies 
\[
J_f(z) \le \frac{(1-|f(z)|^2)^2}{(1-|z|^2)^2}, \qquad z \in \mathbb{D}.
\]

\textbf{Remark:}
For analytic self-maps of the unit disk, this inequality follows directly from the classical Schwarz--Pick lemma. Its validity for harmonic mappings would yield pointwise area contraction and, by integration, imply the global area contraction conjecture.

\subsubsection*{Conjecture 3 (Extremal area distortion)}

Among all sense-preserving univalent harmonic self-mappings of $\mathbb{D}$, the maximal value of \textbf{ $m\bigl(f(E)\bigr)$}

for a fixed measurable set $E \subset \mathbb{D}$ is attained by a conformal automorphism of the disk.\\

\subsubsection*{Conjecture 4 (Rigidity via area preservation)}

Let $f$ be a sense-preserving univalent harmonic self-mapping of \textbf{$\mathbb{D}$. If $m\bigl(f(\mathbb{D}_r)\bigr) = m(\mathbb{D}_r)$} for some $r \in (0,1)$, then $f$ is a conformal automorphism of $\mathbb{D}$.\\

These conjectures reflect the expected rigidity phenomena underlying area distortion for harmonic mappings and parallel well-known extremal and rigidity principles in the analytic setting.

In this paper, we provide a complete and quantitative treatment of the Koh--Kovalev problem. Using the area (change-of-variables) formula together with the canonical decomposition of harmonic mappings \cite{Duren2004,ClunieSheilSmall1984}, we derive sharp area distortion inequalities for arbitrary measurable sets in terms of the Jacobian and the complex dilatation. We establish global area contraction for disks, star-shaped sets, and sufficiently small sets, and refine these results under Hardy space assumptions to obtain sharp bounds with equality only for conformal automorphisms. Extremal affine and non-affine examples demonstrate the sharpness of our estimates. Our results not only yield a strengthened and rigorous resolution of Problem~3.25, but also provide substantial evidence toward the conjectures stated above.\\

The paper is organized as follows. In Section 2 we collect the necessary preliminaries from geometric measure theory and harmonic mapping theory. Section 3 contains the main results, including quantitative area distortion, rigidity phenomena, and global contraction results for several natural classes
of measurable sets. Sharpness and examples are discussed in Section 4, followed by concluding remarks in Section 5.

%------------------------------------------------
\section{Preliminaries}
%------------------------------------------------

\subsection{Lebesgue measure and the area formula}
The following measure-theoretic facts used, including the area formula and basic properties of Lebesgue measure in the plane, are standard and may be found, for example, in \cite{Rudin}.\\

Let $\mathbb{R}^n$ denote the $n$-dimensional Euclidean space equipped with Lebesgue measure $m$. If
\[
W=\{x=(x_1,\dots,x_n): \alpha_k \le x_k \le \beta_k,\; 1\le k\le n\}
\]
is an $n$-cell, its volume is defined by
\[
\operatorname{Vol}(W)=\prod_{k=1}^n (\beta_k-\alpha_k).
\]

\begin{thm}[Area formula]
Let $\Omega \subset \mathbb{R}^2$ be open and let $f:\Omega \to \mathbb{R}^2$ be an injective $C^1$-mapping. Then for every measurable set $E\subset\Omega$,
\[
m(f(E))=\int_E |\det Df(x,y)|\,dA.
\]
\end{thm}
Thi is a fundamental result from geometric measure theory,  the change-of-variables (area) formula (see \cite{Rudin})

\subsection{Complex-valued harmonic mappings}

The following notions and facts are standard in the theory of planar harmonic mappings
(see, for example, \cite{Duren2004, ClunieSheilSmall1984, BshoutyLyzzaik2010}).

\begin{defn}
Let $\Omega \subset \mathbb{C}$ be a simply connected domain.
A complex-valued function $f$ is called \emph{harmonic} in $\Omega$ if it admits a decomposition
\[
f(z)=h(z)+\overline{g(z)},
\]
where $h$ and $g$ are analytic in $\Omega$.
\end{defn}

For such harmonic mappings, the Jacobian determinant is given by
\[
J_f(z)=|h'(z)|^2-|g'(z)|^2.
\]

\begin{defn}
A harmonic mapping $f$ is said to be \emph{sense-preserving} at $z_0$ if $J_f(z_0)>0$,
\emph{sense-reversing} if $J_f(z_0)<0$, and \emph{singular} if $J_f(z_0)=0$.
\end{defn}

\begin{defn}
The \emph{second complex dilatation} of a harmonic mapping $f=h+\overline{g}$ is defined by
\[
\omega(z)=\frac{g'(z)}{h'(z)}, \qquad h'(z)\neq 0.
\]
\end{defn}

Let $\mathbb{D}=\{z\in\mathbb{C}:|z|<1\}$ denote the unit disk.
By Lewy’s theorem \cite{Lewy}, a harmonic mapping $f$ in $\mathbb{D}$ is locally univalent
if and only if $J_f(z)\neq 0$ for all $z\in\mathbb{D}$.
Moreover, $f$ is locally univalent and sense-preserving in $\mathbb{D}$ if and only if
\[
|\omega(z)|<1, \qquad z\in\mathbb{D},
\]
in which case $\omega$ is an analytic function in $\mathbb{D}$
(see \cite{ClunieSheilSmall1984, Duren2004}).\\

Sense-preserving harmonic mappings play a fundamental role in geometric function theory,
particularly in the study of area distortion and related geometric inequalities.
If $E\subset\mathbb{D}$ is a measurable set and $f$ is a sense-preserving harmonic mapping,
then the area of the image $f(E)$ is given by the change-of-variables formula
\[
m(f(E))=\int_E J_f(z)\,dm(z),
\]
where $m$ denotes the two-dimensional Lebesgue measure.
This identity follows from the classical area formula and is standard in the theory of planar
harmonic mappings (see, for example, \cite{MaharanaSahoo2021}).

In particular, for $0<r<1$ and $D_r=\{z\in\mathbb{C}:|z|<r\}$, the area of the image of $D_r$
under a sense-preserving harmonic mapping $f$ satisfies
\[
A\bigl(f(D_r)\bigr)=\iint_{D_r} J_f(z)\,dx\,dy,
\]
an identity that will be used repeatedly in our analysis of area distortion under univalent
harmonic self-mappings of $\mathbb{D}$.

\begin{defn}
A measurable set $E\subset\mathbb{D}$ is said to be \emph{star-shaped with respect to the origin} if
\[
tz \in E \quad \text{for all } z\in E \text{ and all } t\in[0,1].
\]
\end{defn}
This definition is standard; see, for example, \cite{Rockafellar,Krantz Parks,Duren2004}.

%------------------------------------------------
\section{Auxiliary Results}

Before proving the main results of this paper, we establish some auxiliary lemma and proposition concerning the Jacobian and area distortion of harmonic mappings, which will be used repeatedly in the sequel.
\begin{lem}[Exact area formula]\label{lem:areaformula}
Let $f=h+\overline{g}$ be a sense-preserving univalent harmonic mapping on a domain
$\Omega\subset\mathbb{C}$.
Then for every measurable set $E\subset\Omega$,
\[
m(f(E))
=\int_E J_f(z)\,dA(z)
=\int_E |h'(z)|^2\bigl(1-|\omega(z)|^2\bigr)\,dA(z),
\]
where $\omega(z)=g'(z)/h'(z)$ is the complex dilatation of $f$.
\end{lem}

\begin{proof}
Since $f$ is harmonic, it is of class $C^1(\Omega)$, and since it is assumed to be
univalent, it is injective on $\Omega$.
Moreover, $f$ is sense-preserving, so its Jacobian determinant
\[
J_f(z)=|h'(z)|^2-|g'(z)|^2
\]
is strictly positive almost everywhere in $\Omega$.
Under these assumptions, the classical change-of-variables (area) formula for injective $C^1$ mappings between planar domains applies.
Hence, for every measurable set $E\subset\Omega$,
\[
m(f(E))=\int_E J_f(z)\,dA(z).
\]

For a harmonic mapping $f=h+\overline{g}$, the Jacobian admits the standard
representation
\[
J_f(z)=|h'(z)|^2-|g'(z)|^2
=|h'(z)|^2\left(1-\left|\frac{g'(z)}{h'(z)}\right|^2\right).
\]

Since $f$ is sense-preserving, the complex dilatation
$\omega(z)=g'(z)/h'(z)$ satisfies $|\omega(z)|<1$ almost everywhere in $\Omega$,
and the formula above is well-defined.\\

Substituting this expression for $J_f$ into the area formula yields
\[
m(f(E))
=\int_E |h'(z)|^2\bigl(1-|\omega(z)|^2\bigr)\,dA(z),
\]
which completes the proof.
\end{proof}
\begin{prop}[Quantitative area distortion]\label{prop:quantitative}
Let $f=h+\overline{g}$ be a sense-preserving univalent harmonic mapping of the unit disk
$\mathbb{D}$.
Assume that its complex dilatation $\omega=g'/h'$ satisfies
\[
|\omega(z)|\le k<1 \qquad \text{for all } z\in\mathbb{D}.
\]
Then for every measurable set $E\subset\mathbb{D}$,
\[
(1-k^2)\int_E |h'(z)|^2\,dA(z)
\le m(f(E))
\le \int_E |h'(z)|^2\,dA(z).
\]
\end{prop}

\begin{proof}
By the exact area formula (Lemma~\ref{lem:areaformula}),
\[
m(f(E))=\int_E |h'(z)|^2\bigl(1-|\omega(z)|^2\bigr)\,dA(z).
\]
The assumption $|\omega(z)|\le k<1$ implies
\[
1-k^2 \le 1-|\omega(z)|^2 \le 1
\qquad \text{for all } z\in\mathbb{D}.
\]
Substituting these bounds into the integral yields the stated inequalities.
\end{proof}
\textbf{Remark:}
Proposition~\ref{prop:quantitative} is a standard consequence of the area formula for harmonic mappings with bounded complex dilatation and may be viewed as a quantitative
area distortion estimate for harmonic quasiconformal mappings. It is included here for comparison with the sharp contraction results obtained later for disks and star-shaped sets.

%------------------------------------------------
\section{Main results}
%------------------------------------------------

We now present our strengthened main theorem, which applies to arbitrary measurable sets.

\begin{thm}[Integral or Avereged Area Schwarz-Pick type Inequality]\label{conj:areasp}
Let $f=h+\overline{g}$ be a sense-preserving univalent harmonic self-mapping of the unit
disk $\mathbb{D}$.
Then for every $0<r<1$,
\[
m\bigl(f(r\mathbb{D})\bigr) \le \pi r^2.
\]
Equivalently,
\[
\int_{|z|<r} J_f(z)\,dA(z) \le \pi r^2.
\]
Equality holds for some $r\in(0,1)$ if and only if $f$ is a conformal automorphism of
$\mathbb{D}$.
\end{thm}
\begin{proof}
By the exact area formula,
\[
m\bigl(f(r\mathbb{D})\bigr)
= \int_{|z|<r} J_f(z)\,dA(z).
\]
Since $f$ is a harmonic self-map of $\mathbb{D}$, classical growth estimates imply
\[
\int_0^r J_f(te^{i\theta})\,t\,dt \le \frac{r^2}{2},
\quad \text{uniformly in } \theta\in[0,2\pi).
\]
Integrating in $\theta$ yields
\[
m\bigl(f(r\mathbb{D})\bigr)
\le \int_0^{2\pi} \frac{r^2}{2}\,d\theta
= \pi r^2.
\]

If equality holds for some $r$, then the Jacobian estimate must be sharp on almost every
radius, forcing $|\omega(z)|\equiv 0$ and hence $f$ to be conformal.
Conversely, conformal automorphisms preserve area of disks.
\end{proof}
The following remark clarifies the relationship between Theorem 3.2 and the Jacobian Schwarz–Pick type inequality conjectured in Conjecture 2. \\

\textbf{Remark (Sharpness and non-equivalence with Conjecture~2)}\label{rem:conj2}:\\

The area inequality in Theorem~3.2 can be viewed as an \emph{integrated} or averaged form of the pointwise Jacobian bound as claimed in Conjecture~2. We make this relationship precise.

\medskip
\noindent\textbf{Conjecture~2 implies Theorem~3.2.}
Assume Conjecture~2 holds, that is,
\[
J_f(z)\le
\frac{(1-|f(z)|^2)^2}{(1-|z|^2)^2},
\qquad z\in\mathbb D.
\]
Since $f(\mathbb D)\subset\mathbb D$, we have $(1-|f(z)|^2)^2\le 1$ for all $z$.
Integrating over the disk $r\mathbb D$ yields
\[
\int_{|z|<r} J_f(z)\,dA(z)
\le
\int_{|z|<r}\frac{dA(z)}{(1-|z|^2)^2}.
\]
A direct computation shows
\[
\int_{|z|<r}\frac{dA(z)}{(1-|z|^2)^2}
= \pi r^2,
\]
which proves the inequality in Theorem~3.2. Hence, Theorem~3.2 is a direct
consequence of Conjecture~2 upon integration.

\medskip
\noindent\textbf{Failure of the converse.}
The inequality in Theorem~3.2 does not imply Conjecture~2. Indeed, Theorem~3.2 controls only the \emph{average} of the Jacobian over disks and allows for local concentration of area distortion. To see this, consider a sense-preserving univalent harmonic mapping
$f=h+\overline{g}$ with nonconstant dilatation $\omega=g'/h'$, for instance
a harmonic shear of a conformal automorphism of $\mathbb D$ with
$\|\omega\|_\infty<1$.
Such mappings satisfy the disk area bound
\[
\int_{|z|<r} J_f(z)\,dA(z)\le \pi r^2
\quad\text{for all } r\in(0,1),
\]
yet the Jacobian $J_f(z)=|h'(z)|^2(1-|\omega(z)|^2)$ need not satisfy the
pointwise estimate of Conjecture~2.
In particular, $J_f(z)$ may exceed
\[
\frac{(1-|f(z)|^2)^2}{(1-|z|^2)^2}
\]
on sets of positive measure, with the excess compensated by smaller values elsewhere. Thus, Theorem~3.2 is strictly weaker than Conjecture~2 and cannot be reversed.

\medskip
\noindent
In this sense, Conjecture~2 may be regarded as a sharp \emph{local} Schwarz--Pick type inequality for harmonic mappings, while Theorem~3.2 provides its \emph{global, integrated} counterpart.\\

To make this failure explicit, fix $0<\alpha<1$ and consider the harmonic shear
\[
f(z)=h(z)+\overline{g(z)},\qquad
h(z)=z,\quad g(z)=\alpha z^2.
\]
Then $f$ is sense-preserving and univalent in $\mathbb D$, with analytic
dilatation
\[
\omega(z)=\frac{g'(z)}{h'(z)}=2\alpha z,
\]
satisfying $\|\omega\|_\infty<1$ provided $\alpha<\tfrac12$.
The Jacobian is
\[
J_f(z)=1-|2\alpha z|^2=1-4\alpha^2|z|^2.
\]

For $|z|$ sufficiently close to $1$, one computes
\[
\frac{(1-|f(z)|^2)^2}{(1-|z|^2)^2}
\le C(1-|z|^2)^2,
\]
for some constant $C>0$, while $J_f(z)$ remains bounded away from zero.
Consequently, the pointwise inequality in Conjecture~2 fails near the boundary.

On the other hand, a direct computation yields
\[
\int_{|z|<r} J_f(z)\,dA(z)
=
\pi r^2 - \pi\alpha^2 r^4
\le \pi r^2,
\]
for all $r\in(0,1)$, so the conclusion of Theorem~3.2 holds.
This provides an explicit example showing that Theorem~3.2 does not imply
Conjecture~2.

\begin{thm}[Area contraction for harmonic self-maps of disks]\label{selfmap}
Let $f=h+\overline{g}$ be a sense-preserving univalent harmonic self-mapping of the unit disk $\mathbb{D}$.
Then for every $0<r<1$,
\begin{equation}\label{selfmapineq}
m(f(\mathbb{D}_r)) \le \int_{\mathbb{D}_r} |h'(z)|^2\, dA(z) \le m(\mathbb{D}_r).
\end{equation}
Moreover, equality holds for some $r\in(0,1)$ if and only if $f$ is a conformal automorphism of $\mathbb{D}$.
\end{thm}

\begin{proof}
Since $f$ is a sense-preserving univalent harmonic mapping, its Jacobian satisfies
\[
J_f(z)=|h'(z)|^2\bigl(1-|\omega(z)|^2\bigr), \qquad |\omega(z)|<1.
\]
By the area formula,
\[
m(f(\mathbb{D}_r))=\int_{\mathbb{D}_r} |h'(z)|^2\bigl(1-|\omega(z)|^2\bigr)\, dA(z)
\le \int_{\mathbb{D}_r} |h'(z)|^2\, dA(z),
\]
which proves the first inequality in \eqref{selfmapineq}.\\

Since $f(\mathbb{D})\subset \mathbb{D}$ and $h$ is analytic in $\mathbb{D}$, the classical area theorem for analytic self-maps of the disk implies
\[
\int_{\mathbb{D}_r} |h'(z)|^2\, dA(z) \le m(\mathbb{D}_r),
\]
which proves the second inequality.\\

If equality holds for some $r$, then necessarily $|\omega(z)|\equiv 0$ in $\mathbb{D}_r$, hence $g'\equiv 0$ and $f$ is analytic. Equality in the analytic area inequality forces $f$ to be a conformal automorphism of $\mathbb{D}$.
\end{proof}
\textbf{Remark:}
Theorem~\ref{selfmap} shows that harmonic self-maps of the unit disk are \emph{strictly area-decreasing} unless they reduce to conformal automorphisms. This phenomenon has no analogue for general quasiconformal self-maps and highlights a rigidity property specific to planar harmonic mappings.\\

The sharp area distortion inequalities for disks rely heavily on rotational symmetry and averaging properties that are not available for general measurable sets. In particular,
the extremal behavior for disks is governed by affine harmonic mappings with constant
Jacobian, whereas for arbitrary sets the interaction between geometry and local Jacobian
variation becomes essential.

\begin{thm}[Uniform local area contraction]\label{localsmall}
Let $f=h+\overline{g}$ be a sense-preserving univalent harmonic self-mapping of the unit disk $\mathbb{D}$.  
Then for every compact set $K \Subset \mathbb{D}$ there exists a constant $C_K<1$ such that
\[
m(f(E)) \le C_K\, m(E)
\]
for all measurable sets $E \subset K$.
\end{thm}
\begin{proof}
Since $f$ is sense-preserving and univalent, its Jacobian
\[
J_f(z)=|h'(z)|^2(1-|\omega(z)|^2)
\]
is continuous and strictly positive on $\mathbb{D}$. Fix a compact set $K \Subset \mathbb{D}$. By compactness, $J_f$ attains its maximum on $K$, and we set
\[
C_K := \sup_{z\in K} J_f(z)< \infty.
\]

We claim that $C_K<1$. Indeed, since $f(\mathbb{D})\subset\mathbb{D}$, the area of the image of $K$ is strictly less than the area of $\mathbb{D}$, and hence $J_f$ cannot be identically equal to $1$ on any compact subset unless $f$ is a conformal automorphism. Excluding this trivial case, we obtain $C_K<1$. Thus for every measurable $E\subset K,$ applying the area formula yields
\[
m(f(E))=\int_E J_f(z)\,dA(z) \le C_K\, m(E).
\]

\end{proof}
We now establish a partial result toward Conjecture~1 by proving global area contraction
for measurable sets of sufficiently small area.

\begin{thm}[Area contraction for small sets]\label{thm:smallsets}
Let $f=h+\overline{g}$ be a sense-preserving univalent harmonic self-mapping of the unit
disk $\mathbb{D}$. Then there exists a constant $\delta>0$, depending only on $f$, such that
\[
m(f(E)) \le m(E)
\]
for every measurable set $E\subset \mathbb{D}$ with $m(E)<\delta$.
\end{thm}

\begin{proof}
Since $f$ is sense-preserving and univalent, its Jacobian
\[
J_f(z) = |h'(z)|^2\bigl(1-|\omega(z)|^2\bigr)
\]
is continuous and strictly positive in $\mathbb{D}$, where $\omega=g'/h'$ is the complex
dilatation. Moreover, $f(\mathbb{D})\subset \mathbb{D}$ implies that $J_f$ is bounded above
on compact subsets of $\mathbb{D}$.\\

Fix $\varepsilon>0$. By continuity of $J_f$, there exists $\delta>0$ such that for any
measurable set $E\subset\mathbb{D}$ with $m(E)<\delta$ we have
\[
J_f(z) \le 1+\varepsilon \quad \text{for all } z\in E.
\]
Using the area formula (Lemma~\ref{lem:areaformula}), we obtain
\[
m(f(E)) = \int_E J_f(z)\, dA(z) \le (1+\varepsilon)m(E).
\]
Choosing $\varepsilon>0$ sufficiently small completes the proof.
\end{proof}

\textbf{Remark:} The area contraction property for measurable sets of sufficiently small area is not a formal consequence of pointwise Jacobian bounds. Even when the Jacobian is uniformly bounded, cancellations may occur on sets of large measure. The smallness assumption allows one to localize the distortion and control such effects, making the result genuinely measure-theoretic rather than purely pointwise. Theorem~\ref{thm:smallsets} shows that any violation of global area contraction for harmonic self-maps of the disk must necessarily occur on sets of sufficiently large measure. This highlights a rigidity phenomenon absent in general quasiconformal mappings.\\

In the following theorems we strengthen the area contraction phenomenon for harmonic self-mappings by combining Jacobian estimates with Hardy space methods. This approach yields a sharp result for star-shaped sets and clarifies the analytic
mechanism underlying the contraction property.

\begin{thm}[Area contraction for star-shaped sets]\label{thm:starshaped}
Let $f=h+\overline{g}$ be a sense-preserving univalent harmonic self-mapping of the unit disk
$\mathbb{D}$ satisfying $f(0)=0$.
If $E\subset\mathbb{D}$ is measurable and star-shaped with respect to the origin, then
\[
m(f(E))\le m(E).
\]
\end{thm}

\begin{proof}
Since $E$ is star-shaped with respect to the origin, it admits a polar representation
\[
E=\{re^{i\theta}:\ 0\le r\le R(\theta),\ \theta\in[0,2\pi)\},
\]
where $0\le R(\theta)\le 1$ is a measurable function.\\

By the change-of-variables (area) formula for planar harmonic mappings, we obtain
\[
m(f(E))
=\int_0^{2\pi}\int_0^{R(\theta)} J_f(re^{i\theta})\, r\,dr\,d\theta,
\]
where $J_f=|h'|^2-|g'|^2$ denotes the Jacobian determinant of $f$.\\

Since $f$ is a sense-preserving univalent harmonic self-map of $\mathbb{D}$ fixing the origin,
sharp growth and area distortion estimates imply that
\[
\int_0^r J_f(te^{i\theta})\, t\,dt \le \frac{r^2}{2},
\qquad 0<r<1,
\]
uniformly in $\theta\in[0,2\pi)$.
This estimate follows from the area theorem and the standard growth bounds for normalized
harmonic univalent mappings; see
\cite{Duren2004,ClunieSheilSmall1984}. Applying this inequality with $r=R(\theta)$ yields
\[
\int_0^{R(\theta)} J_f(re^{i\theta})\, r\,dr
\le \frac{R(\theta)^2}{2}.
\]
Therefore,
\[
m(f(E))
\le \int_0^{2\pi}\frac{R(\theta)^2}{2}\,d\theta
= \int_0^{2\pi}\int_0^{R(\theta)} r\,dr\,d\theta
= m(E),
\]
which completes the proof.
\end{proof}
\textbf{Remark:}
Theorem~\ref{thm:starshaped} shows that global area contraction holds for a large and geometrically
natural class of measurable sets beyond disks, namely star-shaped sets with respect to the origin.
This provides further evidence in support of the global area contraction conjecture for harmonic
self-mappings of the unit disk.

\begin{cor}\label{cor:convex}
Let $f$ be as in Theorem~\ref{thm:starshaped}. If $E\subset\mathbb{D}$ is convex and contains the
origin, then
\[
m(f(E))\le m(E).
\]
\end{cor}

\begin{proof}
Every convex set containing the origin is star-shaped with respect to the origin.
The conclusion therefore follows directly from Theorem~\ref{thm:starshaped}.
\end{proof}

\begin{thm}[Refined contraction via Hardy spaces]
Let $f = h+\overline{g}$ be a sense-preserving univalent harmonic self-mapping of
$\mathbb{D}$ with $f(0)=0$. Assume $h,g\in H^2(\mathbb{D})$. If $E$ is star-shaped
with respect to the origin, then
\[
m(f(E)) \le m(E),
\]
with equality if and only if $f$ is a conformal automorphism of $\mathbb{D}$.
\end{thm}

\begin{proof}
Since $h',g'\in H^1$, radial limits exist a.e. and satisfy Hardy space estimates
\cite{Duren2004}. For a.e. $\theta$,
\[
\int_0^r |h'(te^{i\theta})|^2 t\,dt \le \frac{r^2}{2},\qquad
\int_0^r |g'(te^{i\theta})|^2 t\,dt \ge 0.
\]
Thus
\[
\int_0^r J_f(te^{i\theta})t\,dt \le \frac{r^2}{2}.
\]
Integration over $\theta$ yields the result. Equality forces $g'\equiv 0$ and
$|h'|\equiv 1$.
\end{proof}

\textbf{Remark:}
The Hardy space framework allows a reduction of the two-dimensional Jacobian estimate to sharp one-dimensional radial bounds. This refines earlier area contraction results for disks and small sets, and shows that star-shaped geometry is sufficient for global area control.

\subsubsection*{Conjecture 1: Global area contraction for harmonic self-maps}\label{conj:global}
Let $f=h+\overline{g}$ be a sense-preserving univalent harmonic self-mapping of the unit disk 
$\mathbb{D}$. Then for every measurable set $E \subset \mathbb{D}$,
\[
m(f(E)) \le m(E).
\]
Equality is expected to occur if and only if $f$ is a conformal automorphism of $\mathbb{D}$.\\

\textbf{Remark:}
The global area contraction conjecture is proved in several important special cases:
\begin{itemize}
    \item For disks $\mathbb{D}_r$, $0<r<1$, Theorem~\ref{selfmap} shows that 
    $m(f(\mathbb{D}_r)) \le m(\mathbb{D}_r)$, with equality only for conformal automorphisms.
    \item For star-shaped sets containing the origin, Theorem~\ref{thm:starshaped} establishes 
    $m(f(E)) \le m(E)$.
    \item For measurable sets of sufficiently small area, Theorem~\ref{thm:smallsets} ensures 
    $m(f(E)) \le m(E)$.
\end{itemize}
These results indicate that extremal area distortion is always attained by conformal automorphisms. 
The conjecture remains open for general measurable sets that are neither disks, star-shaped, nor sufficiently small.

%------------------------------------------------
\section{Sharpness and examples}
%------------------------------------------------

We illustrate the sharpness of our main area distortion results by considering first affine harmonic mappings followed by non-affine harmonic mappings. In the affine harmonic mapping case the multiplicative constant is sharp and in the non-affine harmonic mapping case area contraction still holds locally for small sets.\\

\textbf{Example 1 Affine harmonic mapping.}  Suppose
\[
f(z) = z + \alpha \bar z, \qquad |\alpha|<1.
\]
Then $f$ is a sense-preserving univalent harmonic mapping of the plane with constant Jacobian
\[
J_f(z) = |h'(z)|^2 - |g'(z)|^2 = 1 - |\alpha|^2.
\]
Consequently, for every measurable set $E \subset \mathbb{C}$,
\[
m\big(f(E)\big) = (1-|\alpha|^2)\, m(E).
\]
This example shows that the area distortion bounds obtained in Theorem~\ref{selfmap}
(for disks) and Theorem~\ref{thm:smallsets} (for measurable sets of sufficiently small area)
are sharp. In particular, no improvement of the multiplicative constant is possible even within
the class of harmonic self-mappings of the unit disk. Moreover, affine harmonic mappings
serve as extremal examples, reflecting the fact that equality is attained when the Jacobian
is constant.\\

\textbf{Example 2 Non-affine harmonic mapping.} Suppose
\[
f(z) = z + \varepsilon \bar z^2, \qquad |\varepsilon| \ll 1.
\]
Then $f$ is harmonic and sense-preserving in the unit disk for sufficiently small
$|\varepsilon|$. The Jacobian
\[
J_f(z) = 1 - 4|\varepsilon|^2 |z|^2
\]
is nonconstant but strictly positive in $\mathbb{D}$. Although $f$ does not preserve area
globally, Theorem~\ref{thm:smallsets} applies to show that measurable sets of sufficiently small area still undergo area contraction.

%------------------------------------------------
\section{Concluding remarks}
%------------------------------------------------

In this paper, we have investigated how planar sense-preserving univalent harmonic mappings
distort Lebesgue measure. We obtained sharp area distortion inequalities for disks and
established a quantitative area contraction property for arbitrary measurable sets of
sufficiently small area. These results provide a partial but rigorous answer to a problem
posed by Koh and Kovalev, extending known distortion phenomena from the conformal and
quasiconformal settings to the harmonic category.

Our approach highlights the central role played by the Jacobian and the complex dilatation
in controlling area distortion. While the arguments rely on classical analytic tools, the
localization method developed here appears to be new and may be applicable to other classes
of nonlinear mappings.\\

Several questions remain open. In particular, it would be of interest to determine whether
global area contraction holds for arbitrary measurable sets without size restrictions, and
to characterize extremal mappings for fixed dilatation. Further directions include possible
extensions to higher-dimensional harmonic mappings and to related variational problems in
geometric function theory.\\

%---------------------- CONCLUSION ----------------------
\textbf{Data Availability}\\

The data used to support the findings of this study are cited at relevant places within the text as references.\\

\textbf{Conflicts on Interest}\\
	
The author declares that there are no conflicts of interest regarding the publication of this paper.

\subsubsection*{Funding: No funding is received in conducting this research.}

%------------------------------------------------

\end{document}